\documentclass[12pt]{amsart}
\usepackage{amsthm,amsfonts,amsmath,amssymb,latexsym,epsfig,mathrsfs,yfonts,marvosym}
\usepackage{graphicx,color}
\usepackage{url}
\usepackage[colorlinks=true]{hyperref}
\usepackage[utf8]{inputenc}
\usepackage[T1]{fontenc}

\newtheorem{lemma}{Lemma}[section]
 
\newtheorem{remark}{Remark}[section]
\newtheorem{theorem}{Theorem}

\newtheorem{example}{Example}
\newtheorem*{ack}{Acknowledgements}

\def\<{\langle}

\def\>{\rangle}

\newcommand{\grad}{\mathop{\mathrm{grad}}\nolimits}

\begin{document}

\title{Strongly Hermitian Einstein-Maxwell Solutions on Ruled Surfaces}
\author{Caner Koca and Christina W. T{\o}nnesen-Friedman}
\thanks{This work was partially supported by a grant from the Simon's Foundation (208799 to Christina T{\o}nnesen-Friedman)}
\address{Caner Koca \\ Department of Mathematics\\New York City College of Technology\\CUNY\\ Brooklyn\\ NY 11201\\ USA } \email{CKoca@citytech.cuny.edu}
\address{Christina W. T{\o}nnesen-Friedman\\ Department of Mathematics\\ Union
College\\ Schenectady\\ New York 12308\\ USA } \email{tonnesec@union.edu}

\begin{abstract}
This paper produces explicit strongly Hermitian Einstein-Maxwell  solutions on the smooth compact $4$-manifolds that are $S^2$-bundles over compact Riemann surfaces of any genus. This generalizes the existence results by C. LeBrun in \cite{LeJGP,Le15}. Moreover, by calculating the (normalized) Einstein-Hilbert functional of our examples we generalize Theorem E of \cite{Le15}, which speaks to the abundance of Hermitian Einstein-Maxwell solutions on such manifolds. As a bonus, we exhibit  certain pairs of strongly Hermitian Einstein-Maxwell solutions, first found in \cite{Le15}, on the first Hirzebruch surface in a form which clearly shows that they are conformal to a common K\"ahler metric. In particular, this yields a non-trivial example of non-uniqueness of positive constant scalar curvature metrics in a given conformal class.
\end{abstract}
\maketitle

\section{Introduction}
It is well known that on a compact Riemannian $4$-manifold $(M,h)$, the scalar curvature of $h$ must be constant if $h$ is part of a solution the \emph{Einstein-Maxwell equations }\cite{Le08};
\begin{equation}\label{em}
\begin{array}{rcl}
dF & = & 0\\
\\
d\star F & = & 0\\
\\
\left[ r + F \circ F\right]_0 &=& 0,
\end{array}
\end{equation}
where $r$ is the Ricci tensor of $h$, $F$ is a real $2$-form on $M$, $[\,]_0$ denotes the trace-free part with respect to $h$, and $F\circ F$ is the  composition of $F$ with itself, when we view $F$ as an endomorphism on the tangent bundle $TM$. While the converse is not true in general (as is e.g. manifested on compact, complex, non-K\"ahlerian surfaces by Proposition 3 in \cite{Le08}), it follows from the work of Lebrun and Apostolov, Calderbank and Gauduchon \cite{ACGat1, Le08, LeJGP, Le15} that if $(M^4,g,J)$ is a K\"ahler manifold and $f>0$ is a real holomorphic potential on $(M,J,g)$ such that $h=f^{-2}g$ has constant scalar curvature, then $(h,F)$ solves \eqref{em}, where $F$ is a unique harmonic $2$-form on $M$ with self-dual part equal to the
K\"ahler form $\omega$. Since in that case both $h$ and $F$ are $J$ invariant, such Einstein-Maxwell solutions are called \emph{strongly Hermitian} \cite{LeJGP}.

In this paper we generalize Theorem D of \cite{Le15};
\begin{theorem}\label{existencetheorem}
Let $(M,J)$ be a minimal ruled surface of the form \newline
${\mathbb P}({\mathcal O} \oplus {\mathcal L}) \rightarrow \Sigma$, where $L\rightarrow \Sigma$ is any  holomorphic line bundle
of non-zero degree, ${\mathcal O}\rightarrow \Sigma$ is the trivial holomorphic line bundle, and
$\Sigma$ is a compact Riemann surface of genus $\mathfrak{g}$. Then there exists an open and non-empty subset $\mathscr K$ of the K\"ahler cone on $(M,J)$ such that each
K\"ahler class in $\mathscr K$ contains a K\"ahler metric $g$ which is conformal to a (non-K\"ahler)  Einstein-Maxwell metric $h$. When $\mathfrak{g} \leq 1$, $\mathscr K$ may be taken to be the entire K\"ahler cone.
\end{theorem}

\begin{remark}
For all values of  $\mathfrak{g}$ we formally get a solution, represented by a  polynomial, for each K\"ahler class.
The positivity of this polynomial over a certain interval is necessary and sufficient for the  formal solution to correspond
to an actual solution for the given K\"ahler class. When $\mathfrak{g}\leq 1$, positivity always hold. When $\mathfrak{g}>1$, there will be K\"ahler classes where positivity holds and K\"ahler classes where positivity fails. This parallels a phenomenon observed for extremal K\"ahler metrics (see e.g. \cite{t-f} or \cite{acgt}).
\end{remark}

The (normalized) Einstein-Hilbert functional evaluated for a Riemannian metric on a compact $4$-manifold $M$  is defined by
$${\mathfrak S}:= \frac{\int_M Scal\, d\mu}{\sqrt{\int_M d\mu}}.$$
Note that this is invariant under re-scaling.
For a detailed introduction to this functional we refer to \cite{Le15, LeJGP}. Here we will just give a very brief summary of the two aspects we will consider.

For a fixed conformal class, $[g]$,  of Riemannian metrics, ${\mathfrak S}\mid_{[g]}$  is the Yamabe functional and the Yamabe constant, $Y_{[g]}$,  of $[g]$ is then defined as the infimum of
${\mathfrak S}\mid_{[g]}$. From the famous work of Yamabe, Trudinger, Aubin, and Schoen \cite{A82, LePa87, S84}, we know that this infimum is in fact achieved by a metric (a \emph{Yamabe minimizer}) in $[g]$.
This metric must have constant scalar curvature and if $Y_{[g]} \leq 0$, the minimizer is unique (up to re-scaling) and is the one and only constant scalar curvature metric in $[g]$. For $Y_{[g]} >0$, the Yamabe minimizer is not necessarily unique and further a constant scalar curvature metric in $[g]$ is not necessarily a Yamabe minimizer. This makes the estimation of $Y_{[g]}$ for the case where $[g]$ has no constant negative scalar curvature representative very difficult. We do know \cite{A82} that
$Y_{[g]} \leq 8\sqrt{6}\pi$ and thus if a constant scalar curvature representative of $[g]$ is a Yamabe minimizer, it must satisfy that
$\frac{\int_M Scal\, d\mu}{\sqrt{\int_M d\mu}} \leq 8\sqrt{6}\pi$. Further, by the work of Schoen \cite{S84}, the inequality is known to be sharp if $(M,[g])$ is not conformal to the $4$-sphere.

We can also consider ${\mathfrak S}$ on the following space. Assume the orientation of $M$ is fixed. Let $\Omega$ be a fixed cohomology class in $H^2(M,{\mathbb R})$ such that $\Omega^2 >0$ and let
$\mathscr{G}_\Omega$ denote the set of smooth Riemannian metrics $h$ on $M$ for which the
harmonic representative $\omega$ of  $\Omega$ is self-dual. In particular, $\mathscr{G}_\Omega$ will include any Riemannian metric that is K\"ahler with respect to some complex structure (compatible with the fixed orientation) on $M$ such that its K\"ahler form belongs to $\Omega$. Obviously if $h \in \mathscr{G}_\Omega$, then $[h] \subseteq \mathscr{G}_\Omega$. The critical points of ${\mathfrak S}|_{\mathscr{G}_\Omega}$ are exactly the Einstein-Maxwell solutions  for which the self-dual part $F^+$ of the 2-form $F$ is in $\Omega$ (cf. Proposition 1 in  \cite{Le15}).

In this setting, as defined by LeBrun \cite{Le15}, the moduli-space of the $\Omega$-compatible solutions of the Einstein-Maxwell equations is
\begin{equation}\label{modulispace}
\mathcal M_{\Omega} = \{(h,F)\textnormal{ solves \eqref{em}} \mid F^{+}\in\Omega\}/[\mathrm{Diff}_H (M)\times \mathbb R^+]
\end{equation}
where $\mathrm{Diff}_H(M)$ is the group of diffeomorphisms of $M$ acting trivially on $H^2(M,\mathbb R)$. $\mathbb R^+$ acts by rescaling the metric $h$, but does not change the 2-form $F$. LeBrun shows that the value of $\mathfrak S$ is an invariant of connected components of $\mathcal M_\Omega$; that is, if $(h,F)$ and $(\tilde h,\tilde F)$ are solutions with $\mathfrak S(h) \neq \mathfrak S(\tilde h)$, then they must belong to different connected components of $\mathcal M_\Omega$ (Proposition 2 in \cite{Le15}). By computing the value of the functional $\mathfrak S$ for the Einstein-Maxwell solutions he found on the Hirzebruch surfaces, whose underlying smooth manifolds are $S^2\times S^2$ or $\mathbb{CP}_2\#\overline{\mathbb{CP}}_2$, LeBrun shows that $\mathcal M_\Omega$ has as many connected components as we wish for an appropriate de Rham class $\Omega$ on these manifolds (cf. Theorem E in \cite{Le15}). In this paper, we also give a generalization of this result. Analogously, the diffeotypes of the complex surfaces $\mathbb P(\mathcal O \oplus \mathcal L)$ of our Theorem \ref{existencetheorem} fall into two cases: They are either a product $S^2 \times \Sigma$, or otherwise the twisted product $S^2\tilde{\times} \Sigma$.

\begin{theorem}\label{modulispacetheorem}
Let the smooth 4-manifold $M$ be either the product $S^2 \times \Sigma$ or the twisted product $S^2\tilde{\times} \Sigma$ where $\Sigma$ is a Riemann surface of genus $\mathfrak g$.  Then, for any given natural number $\mathbf{N}$ we can find a de Rham class $\Omega$ on $M$ with $\Omega^2>0$ such that the moduli space $\mathcal M_\Omega$ given by \eqref{modulispace} has at least $\mathbf{N}$ components.
\end{theorem}

The idea behind the proof of Theorem \ref{modulispacetheorem} is to first fix the de Rham class $\Omega$ on $M$, but to allow the complex structure to vary by considering the line bundles $\mathcal L$ of different degrees. This allows us to view the same de Rham class $\Omega$ as the K\"ahler class with respect to different complex structures on $M$. The admissible K\"ahler metrics in $\Omega$ for these different complex structures will then yield Einstein-Maxwell metrics after conformal rescaling by Theorem \ref{existencetheorem}. The calculation of the Einstein-Hilbert functional $\mathfrak S$, which is presented in Sections \ref{EHg=1} and \ref{EHgengeq2}, show that we get as many different values as we wish by changing our initial choice of the de Rham class $\Omega$.

The outline of the paper is as follows: In Section \ref{hirze}, we review the construction of \emph{admissible K\"ahler metrics} on the minimal ruled surfaces $\mathbb{P}(\mathcal O\oplus\mathcal L)$. In Section \ref{admisEMsoln}, we look at a particular conformal rescaling of the admissible K\"ahler metrics by positive \emph{holomorphy potentials}, and determine when this rescaling is a solution of the Einstein-Maxwell equations. The special phenomenon, mentioned in the abstract, on the first Hirzebruch surface is treated in Section \ref{Case2}. In Section \ref{EHsection}, we compute the value of the Einstein-Hilbert functional $\mathfrak S$ for the Einstein-Maxwell solutions we find in Section \ref{admisEMsoln} and discuss the Yamabe constant of the conformal classes of these metrics. Finally, we end the paper with a short section, Section \ref{amf}, discussing the connections with the recent results of Vestislav Apostolov and Gideon Maschler \cite{am}.

\begin{ack}
We would like to warmly thank Claude LeBrun for his advice and  encouragement  as we were writing this paper. We would also like to thank Vestislav Apostolov and Gideon Maschler for their insightful comments. Further, we are grateful to the anonymous referee for making some very helpful comments that resulted in the addition of Section \ref{amf}.
\end{ack}

\section{Admissible metrics on Ruled Surfaces}\label{hirze}

Let $S_n$ be a ruled surface of the form
${\mathbb P}({\mathcal O} \oplus {\mathcal L}_n) \rightarrow \Sigma$,
where $\Sigma$ is a compact Riemann surface, ${\mathcal L}_n$ is a
holomorphic line bundle of degree $n\in {\mathbb Z}^+$ on $\Sigma$, and
${\mathcal O}$ is the trivial holomorphic line bundle. We will call this an \emph{admissible ruled surface} \cite{acgt}.

Since $({\mathcal O} \oplus {\mathcal L}_n) \rightarrow \Sigma$ is not a polystable holomorphic vector bundle we know that $S_n$ admits no cscK metrics
(see e.g. Theorem 2 in \cite{at}).
Let $g_{\Sigma}$ be the
K\"ahler metric
on $\Sigma$ of constant scalar curvature $2s_{\Sigma}$, with K\"ahler form
$\omega_{\Sigma}$, such that
$c_{1}({\mathcal L}_n) = [\frac{\omega_{\Sigma}}{2 \pi}]$.
Let ${\mathcal K}_\Sigma$ denote the canonical bundle of $\Sigma$. Since $c_1({\mathcal K}_\Sigma^{-1}) = [\rho_\Sigma/2\pi]$, where $\rho_\Sigma$ denotes the Ricci form, we must have that $s_{\Sigma}= 2(1-\mathfrak{g})/n$, where $\mathfrak{g}$
denotes the genus of $\Sigma$. In particular, note that $s_\Sigma \leq 2$.

The natural $\mathbb{C}^*$-action on ${\mathcal L}_n$
extends to a holomorphic
$\mathbb{C}^*$-action on $S_n$. The open and dense set $S_n^0$ of stable points with
respect to the
latter action has the structure of a principal $\mathbb{C}^*$-bundle over the stable
quotient.
A choice of a hermitian norm on the fibers of ${\mathcal L}_n$, whose Chern connection has curvature equal to $\omega_\Sigma$, induces via a Legendre transform a function
$\mathfrak{z}:S_n^0\rightarrow (-1,1)$ whose extension to $S_n$ consists of the critical manifolds
$E_{0}:=\mathfrak{z}^{-1}(1)=P({\mathcal O} \oplus 0)$ and
$E_{\infty}:= \mathfrak{z}^{-1}(-1)=P(0 \oplus {\mathcal L}_n)$.
These  are respectively the zero and infinity section of $S_n \rightarrow
\Sigma$. It is well-known that  $E_{0}$ and $E_{\infty}$ have the property that $E_{0}^{2} = n$ and $E_{\infty}^{2} = -n$,
respectively. If $C$ denotes a fiber of the ruling $S_n \rightarrow
\Sigma$, then $C^{2}=0$, while $C \cdot E_{i} =1$ for both, $i=0$ and
$i=\infty$. Any real cohomology class in the two dimensional space
$H^{2}(S_n, {\mathbb R})$ may be written as a linear combination of
(the Poincar\'e
duals of) $E_{0}$ and $C$,
$$
m_{1} E_{0} +m_{2} C\, .
$$
Thus, we may think of $H^{2}(S_n, {\mathbb R})$ as ${\mathbb R}^2$,
with coordinates $(m_{1},m_{2})$. The K\"ahler cone ${\mathcal K}$ may
be identified with \newline
${\mathbb R}_{+}^2=\{ (m_{1}, m_{2})\,  | \; m_{1} > 0,
m_{2} > 0 \}$ (see \cite{fujiki} or Lemma 1 in \cite{t-f}).

The \emph{admissible K\"ahler metrics} were introduced as such in \cite{acgt} . What follows is a quick overview on how to build such metrics
on $S_n$. We will use the notation from \cite{acgt}, but it is fair to note that in the special case of $S_n$  this construction dates further back.
We refer to \cite{acgt} for references as well as more technical details on what follows below.

Let  $\theta$ be a connection one form for the
Hermitian metric on $S_n^0$, with curvature
$d\theta = \omega_\Sigma$. Let $\Theta$ be a smooth real function with
domain containing
$(-1,1)$. Let $x$ be a real number such that $0 < x < 1$.
Then an admissible K\"ahler metric
is given on $S_n^0$ by
\begin{equation}\label{metric}
g  =  \frac{1+x \mathfrak{z}}{x} g_\Sigma
+\frac {d\mathfrak{z}^2}
{\Theta (\mathfrak{z})}+\Theta (\mathfrak{z})\theta^2\,
\end{equation}
with K\"ahler form
\begin{equation*}
\omega =  \frac{1+x \mathfrak{z}}{x}\omega_\Sigma
+d\mathfrak{z}\wedge \theta\,. \label{kf}
\end{equation*}
The complex structure yielding this
K\"ahler structure is given by the pullback of the base complex structure
along with the requirement
\begin{equation}\label{complex}
Jd\mathfrak{z} = \Theta \theta
\end{equation}
 The function $\mathfrak{z}$ is
hamiltonian
with $K= J\grad \mathfrak{z}$ a Killing vector field. Observe that $K$
generates the circle action which induces the holomorphic
$\mathbb{C}^*$- action on $S_n$ as introduced above.
In fact, $\mathfrak{z}$ is the moment
map on $S_n$ for the circle action, decomposing $S_n$ into
the free orbits $S_n^0 = \mathfrak{z}^{-1}((-1,1))$ and the special orbits
$\mathfrak{z}^{-1}(\pm 1)$. Finally, $\theta$ satisfies
$\theta(K)=1$.
In order that $g$ (be a genuine metric and) extend to all of $S_n$,
$\Theta$ must satisfy the positivity and boundary
conditions
\begin{align}
\label{positivity}
(i)\ \Theta(\mathfrak{z}) > 0, \quad -1 < \mathfrak{z} <1,\quad
(ii)\ \Theta(\pm 1) = 0,\quad
(iii)\ \Theta'(\pm 1) = \mp 2.
\end{align}
It is convenient to define a function $F(\mathfrak{z})$ by the formula
\begin{equation}\label{theta}
\Theta(\mathfrak{z})= \frac{F(\mathfrak{z})}{(1+x
\mathfrak{z})}
\end{equation}
Since $(1+x
\mathfrak{z})$ is positive for $-1<\mathfrak{z}<1$, conditions
\eqref{positivity}
imply the following equivalent conditions on $F(\mathfrak{z})$:
\begin{align}
\label{positivityF}
(i)\ F(\mathfrak{z}) > 0, \quad -1 < \mathfrak{z} <1,\quad
(ii)\ F(\pm 1) = 0,\quad
(iii)\ F'(\pm 1) = \mp 2(1 \pm x).
\end{align}

The construction of admissible K\"ahler metrics is based on the symplectic viewpoint. Different choices of $F$ yield different complex structures that are all compatible with the same fixed symplectic form $\omega$. However, for each $F$ there is an $S^1$-equivariant diffeomorphism pulling back $J$ to the original fixed complex structure on $S_n$ in such a way that the K\"ahler form of the new K\"ahler metric is in the same cohomology class as $\omega$ \cite{acgt}. Therefore, with all else fixed, we may view the set of the functions $F$ satisfying \eqref{positivityF} as parametrizing a certain family of K\"ahler metrics within the same K\"ahler class of $S_n$.

One easily checks that the K\"ahler class of an admissible metric \eqref{metric}
\begin{equation}\label{class1}
[\omega] = 4\pi E_{0}+ \frac{2\pi(1-x)n}{x} C
\end{equation}
and hence, up to an overall rescale, every K\"ahler class in the K\"ahler cone may be represented by an admissible K\"ahler metric.

Let $g$ be an admissible metric as above determined by a given $F$. Two observations that may be found in \cite{acg} will be useful to us:
\begin{itemize}
\item The scalar curvature
is given by
\begin{equation}\label{Scal}
 Scal(g) = \frac{2s_{\Sigma} x}{1+x\mathfrak{z}} - \frac{F''(\mathfrak{z})}{1+x\mathfrak{z}}\, .
 \end{equation}
\item If $p(\mathfrak{z})$ is a smooth function of $\mathfrak{z}$,
then
\begin{equation}
\label{Lapl}
\Delta p = -[F(\mathfrak{z}) p'(\mathfrak{z})]'/(1+x\mathfrak{z}),
\end{equation}
where $\Delta$ is the Laplacian associated to $g$.
\end{itemize}

\section{Admissible K\"ahler metrics conformal to solutions of the Einstein-Maxwell Equations}\label{admisEMsoln}

It is clear that for a given constant $b$, the function $\mathfrak{z} + b$ is a real holomorphy potential. Indeed if $\mathfrak{g} \geq 1$, this is, up to rescale and automorphism, the only type of real holomorphy potentials we have. Also, $(\mathfrak{z} + b)^2$ is positive as a function on $S_n$ iff $|b|>1$.

Let $g$ be an admissible metric given by \eqref{metric} and consider the new non-K\"ahler metric $$h= (\mathfrak{z} + b)^{-2} g,$$ where we require that $|b|>1$.
Using the conformal change formula for scalar curvature and the formulas \eqref{Scal} and \eqref{Lapl} above, we calculate that the scalar curvature of $h$ is given by

$$ Scal(h) = \frac{-({\mathfrak{z}}+b)^2F''({\mathfrak{z}}) + 6({\mathfrak{z}}+b) F'({\mathfrak{z}}) - 12 F({\mathfrak{z}}) + 2 s_{\Sigma} x ({\mathfrak{z}}+b)^2}{(1+x {\mathfrak{z}})}.$$
We want this to be equal to a constant, $A$, so by the exact same argument as on page 7 below (12) in \cite{LeJGP}, we must have that $F({\mathfrak{z}})$ is a quartic.\footnote{Here we just consider the linear operator $y \mapsto ({\mathfrak{z}}+b)^2y''- 6({\mathfrak{z}}+b) y' + 12 y$}
Due to (ii) and (iii) in \eqref{positivityF}, this forces $F$ to have the following form:
\begin{equation}\label{F}
F(\mathfrak{z}) = (1-\mathfrak{z}^2)\left( (1+x \mathfrak{z}) - c(1-\mathfrak{z}^2) \right).
\end{equation}
Plugging this into the formula for $Scal(h)$ above and setting the result equal to $A$, we get the following equations:
\begin{equation}\label{A}
Scal(h) = A= \frac{6 \left(1-6 b^2+b^4+2 b x+2 b^3 x-s_\Sigma x+b^4 s_\Sigma x\right)}{3 b^2-1}
\end{equation}

\begin{equation}\label{c}
c= \frac{-1+3 b x-s_\Sigma x}{2 \left(3 b^2-1\right)}
\end{equation}
and
\begin{equation}\label{b}
\left(xb^2 -2 b +x \right) \left((s_\Sigma x-2) b^2+2 b x-s_\Sigma x \right)=0.
\end{equation}
Thus we see that, for a given $0<x<1$, our job is to solve \eqref{b} for $|b| >1$.  Then, for any such solution we need to check if $F$ from \eqref{F} with the associated $c$ from \eqref{c} satisfies (i) of
\eqref{positivityF}. This in turn amounts to checking if $m(\mathfrak{z}):=\left( (1+x \mathfrak{z}) - c(1-\mathfrak{z}^2) \right)$ is positive for $-1< \mathfrak{z} <1$.

\subsection{Case 1: $\underline{xb^2 -2 b +x=0}$:}\label{Case1}
We first note that \eqref{b} is solved when $xb^2 -2 b +x=0$ and this equation has precisely one solution satisfying $|b|>1$, namely
\begin{equation}\label{Case1soln}
b=\frac{1+\sqrt{1-x^2}}{x}
\end{equation}
 (which indeed gives us $b>1$ since $0<x<1$).
Substituting this $b$ into \eqref{c}, we get
\begin{equation}\label{Case1c}
c=\frac{x^2 \left(2-s_\Sigma x+3 \sqrt{1-x^2}\right)}{4 \left(3-2 x^2+3 \sqrt{1-x^2}\right)}
\end{equation}
Now $m(\mathfrak{z})$ is concave up and
$m(-1) =1-x >0$ while
$$m'(-1) = x-2c = \frac{x \left(6-2 x-4 x^2+s_\Sigma x^2+(6-3x) \sqrt{1-x^2}\right)}{2 \left(3-2 x^2+3 \sqrt{1-x^2}\right)}.$$
When $s_\Sigma \geq 0$, $m'(-1) >0$ and so $m(\mathfrak{z})>0$ for $-1< \mathfrak{z} <1$. Thus
in this case, (i) of \eqref{positivityF} is satisfied for any $0<x<1$ and we have our desired metrics for every K\"ahler class in the K\"ahler cone.
For $s_\Sigma>0$, these metrics, which live on the Hirzebruch surfaces, already appear in \cite{Le15}. For $s_\Sigma=0$, we have the case of
$\Sigma$ being the torus, $T^2$.

Further, for any value of $s_\Sigma$, if $0<x<1$ is sufficiently small, we still have that
$m'(-1) >0$ and so $m(\mathfrak{z})>0$ for $-1<\mathfrak{z} <1$. In particular, we have examples for $\Sigma$ being a Riemann surface of \emph{any} genus.
This completes the proof of Theorem \ref{existencetheorem}.

On the other hand, assume $s_\Sigma <0$ is fixed. One observes easily that $-1<\frac{2+s_\Sigma}{s_\Sigma -2} <1$. Let $\mathfrak{z}_0$ be a fixed value such that $-1< \mathfrak{z}_0< \frac{2+s_\Sigma}{s_\Sigma -2}$.  Now
$$\lim_{x\rightarrow 1}m(\mathfrak{z}_0) = \frac{1}{4}(1+\mathfrak{z}_0) ((2-s_\Sigma)\mathfrak{z}_0 + (2+s_\Sigma)) <0,$$
and thus for $x$ sufficiently close to $1$, $m(\mathfrak{z}_0)<0$ and hence (i) of \eqref{positivityF} fails.

We can be a bit more specific about when we have failure of positivity of $m({\mathfrak{z}})$ over the interval $(-1,1)$ in the case of $s_\Sigma<0$. Indeed, it is not hard to confirm the following observations for this case:
\begin{itemize}
\item The graph of $m({\mathfrak{z}})$ is a concave up parabola (as we already noted)
\item $m'(1) = x+ 2c $ is positive for all $0<x<1$, since $c$ is positive for all $0<x<1$
\item $m'(-1)= \frac{x g(x)}{2 \left(3-2 x^2+3 \sqrt{1-x^2}\right)}$, where \newline
$g(x) = \left(6-2 x-4 x^2+s_\Sigma x^2+(6-3x) \sqrt{1-x^2}\right)$ is a monotone decreasing function over the interval $[0,1]$ such that $g(0) = 12 >0$ and $g(1)=s_\Sigma <0$. Thus, there exists a unique value $x_{s_\Sigma,1} \in (0,1)$ so that
\begin{itemize}
\item For $0<x\leq x_{s_\Sigma,1}$, $m'(-1)>0$ and so $m({\mathfrak{z}})$ is positive over the interval $(-1,1)$.
\item For $x=x_{s_\Sigma,1}$, $m'(-1)=0$ and hence $m({\mathfrak{z}})$ has a global positive minimum (of $1-x$) at ${\mathfrak{z}}=-1$.
\item For $x_{s_\Sigma,1}<x<1$, $m'(-1)<0$ and so $m({\mathfrak{z}})$ is positive over the interval $(-1,1)$ if and only if the discriminant of  $m({\mathfrak{z}})$ is negative (i.e.,  $m({\mathfrak{z}})$ has no roots).
\end{itemize}
\item For a given $x\in (0,1)$, the discriminant of  $m({\mathfrak{z}})$ is given by
$$\frac{x^2D_{s_\Sigma}(x) }{4 \left(3-2 x^2+3 \sqrt{1-x^2}\right)^2},$$ where
$$D_{s_\Sigma}(x) = 12 +12s_\Sigma x - 19 x^2 - 12 s_\Sigma x^3 + (7+{s_\Sigma}^2)x^4 + 6\sqrt{1-x^2}(2+2 s_\Sigma x - 2 x^2 - s_\Sigma x^3)$$
with
\begin{itemize}
\item (formally) $D_{s_\Sigma}(0)= 24 >0$
\item $D_{s_\Sigma}(x_{s_\Sigma,1}) <0$ (obviously)
\item (formally) $D_{s_\Sigma}(1) = {s_\Sigma}^2 >0$
\end{itemize}
\item If we define $d_{s_\Sigma}(x) = D_{s_\Sigma}''(x)$ one may check that
\begin{itemize}
\item (formally) $d_{s_\Sigma}(0)<0$
\item $\displaystyle \lim_{x\rightarrow 1^-}d_{s_\Sigma}(x)>0$
\item $d_{s_\Sigma}(x)$ is monotone increasing over the interval $[0,1]$
\end{itemize}
and so $D_{s_\Sigma}(x)$ has exactly one inflection point (changing from concave down to concave up) over the interval $(0,1)$
\item Thus $D_{s_\Sigma}(x)$ has exactly two roots over the interval $(0,1)$; one occurring in $(0,x_{s_\Sigma,1})$ and one occurring in $(x_{s_\Sigma,1},0)$.
Let us denote the latter root by $x_{s_\Sigma,2}$.
We make the following conclusions
\begin{itemize}
\item For $0<x< x_{s_\Sigma,2}$, positivity of $m({\mathfrak{z}})$ over the interval $(-1,1)$ holds.
\item For $ x_{s_\Sigma,2}<x<1$, positivity of $m({\mathfrak{z}})$ over the interval $(-1,1)$ fails.
\end{itemize}
\end{itemize}

Thus, for $s_\Sigma<0$, there exists a unique value $x_{s_\Sigma,2} \in (0,1)$ such that for $0<x< x_{s_\Sigma,2}$, (i) of \eqref{positivityF} is
satisfied and hence we have our desired metrics and if $x_{s_\Sigma,2}<x<1$, (i) of \eqref{positivityF} is not satisfied and hence we do not have our special types of metrics.

The following, somewhat obscure, little lemma will be useful in Section \ref{EHgengeq2}.

\begin{lemma}\label{littlelemma}
For any $s_\Sigma<0$, $x_{s_\Sigma,2} > \frac{1}{s_\Sigma^2 +2}$.
\end{lemma}

\begin{proof}
This follows from the simple fact that at $x=\frac{1}{s_\Sigma^2 +2}$, the value of the monotone decreasing function \newline
$g(x) = \left(6-2 x-4 x^2+s_\Sigma x^2+(6-3x) \sqrt{1-x^2}\right)$ from above is still positive and so
$ \frac{1}{s_\Sigma^2 +2}< x_{s_\Sigma,1} < x_{s_\Sigma,2}$.
\end{proof}

\begin{example}\label{ex1}
Let us assume that $n=1$ and the genus of $\Sigma$, $\mathfrak{g}=2$. Then $s_{\Sigma}= -2$ and the discriminant of $m({\mathfrak{z}})$ is given by
$$D_{-2}(x) = 12 -24 x - 19 x^2 +24 x^3 + 11 x^4 + 12\sqrt{1-x^2}(1-2 x - x^2 + x^3).$$
One may check that numerically $ x_{s_\Sigma,2} \approx 0.97367$ ( while $ x_{s_\Sigma,1}$ from above is about $0.93578$).
\end{example}

As the above discussion shows, just as is the case for extremal K\"ahler metrics, we seem to run into a case where we are not able to exhaust the entire K\"ahler cone with these special admissible metrics. It would be very interesting to find out exactly what happens in the ``bad'' K\"ahler classes; are there no K\"ahler metrics conformal to metrics solving the Einstein-Maxwell equations, or are there just none of this particular type we are seeking above?

\begin{remark}
Note that the only case, where a metric $g$ as above is actually extremal and $h$ is Einstein, is the case where $s_\Sigma=2$ (so $M$ is the first Hirzebruch surface) and $x$ is a certain specific value (it is not a pretty number, but it is explicit). This is of course just the Page metric all over again and was treated in \cite{Le15}.

However, if one allows $0<b<1$, and hence allow for the case where $h= (\mathfrak{z} +b)^{-2}g$ is not defined along the sub manifold $\mathfrak{z}^{-1}(-b)$, there is an extra solution (for any $0<x<1$ such that $-3+2 x^2+3 \sqrt{1-x^2} \neq  0$) with
$$b= \frac{1-\sqrt{1-x^2}}{x}$$
and
$$F(\mathfrak{z}) = \frac{(1-\mathfrak{z}^2)m(\mathfrak{z})}{4 \left(-3+2 x^2+3 \sqrt{1-x^2}\right)},$$
where
$$
\begin{array}{ccl}
m(\mathfrak{z}) &= &  -12+10 x^2-s_\Sigma x^3+(12-3 x^2) \sqrt{1-x^2}\\
\\
&+ & \left(-12 x +8 x^3 +12 x \sqrt{1-x^2} \right) \mathfrak{z}\\
\\
&+ & \left( x^2(s_\Sigma x-2) +3 x^2 \sqrt{1-x^2} \right) \mathfrak{z}^2.
\end{array}
$$
One can check directly that when $s_\Sigma <1$ (so first and second Hirzebruch surface are avoided), there is precisely one value $0<x_0<1$ such that the corresponding $F(\mathfrak{z})$ defines an extremal K\"ahler metric $g$. For this $x_0$,  $h$ will be an Einstein metric defined away from the sub manifold $\mathfrak{z}^{-1}(\frac{\sqrt{1-x_0^2}-1}{x_0})$. This is in line with an observation made in e.g.
Proposition 3 of \cite{t-f02}.
\end{remark}

\subsection{Case 2: $\underline{(s_\Sigma x-2) b^2+2 b x-s_\Sigma x =0}$:}\label{Case2}

Since $s_\Sigma x <2$ we have that $p(b) := (s_\Sigma x-2) b^2+2 b x-s_\Sigma x$ is a concave down parabola such that $p(\pm 1) = \pm 2(x\mp 1) <0$ and
 $p'(\pm 1) = \pm 2(s_\Sigma x - 2 \pm x)$, we easily observe that for $s_\Sigma <2$ (hence $s_\Sigma \leq 1$), there are no solutions to $p(b)=0$ with $|b|>1$.

However when $s_\Sigma=2$, the equation becomes
$$(x-1) b^2+ x b  - x =0.$$
As long as $x > 4/5$ this gives us two solutions
$$ b=\frac{x\pm \sqrt{x(5x - 4)}}{2(1- x)}.$$
When $x=4/5$ we get the same solution as in Case 1 and as $x$ spans $[\frac{4}{5},1)$, we see that $b_1:=\frac{x + \sqrt{x(5x - 4)}}{2(1- x)}$ increases monotonically from $2$ to $+\infty$ while $b_2 := \frac{x - \sqrt{x(5x - 4)}}{2(1- x)}$ decreases monotonically from $2$ to $1$. Thus $b_i >1$ for $i=1,2$.

Now, with $s_\Sigma=2$ and $(x-1) b^2+ x b  - x =0$ we get from \eqref{c} that
$$ c=  \frac{-1+3 b x-2 x}{2 \left(3 b^2-1\right)} = (1-x)/2.$$
Thus for a given $x \in (4/5,1)$, $F(\mathfrak{z})$ is the same for $b_1$ and $b_2$. Similarly to above, it is easy to check that (i) of \eqref{positivityF} is satisfied in this case. So, we have one K\"ahler metric with two different EM solutions $h_1 = (\mathfrak{z} + b_1)^{-2} g$ and $h_2= (\mathfrak{z} + b_2)^{-2}g$.

\begin{remark}
One may check that the condition $x=4/5$ corresponds to the condition $u/v=9$ in Theorem B of \cite{Le15}, so we are certainly just recasting the bifurcation observed by LeBrun. What we discover here is that  $h_1$ and $h_2$ are conformal to one and the same admissible K\"ahler metric. In particular, $h_1$ and $h_2$ belong to the same conformal class.
One may confirm this very surprising observation directly from \cite{Le15}:

The K\"ahler metric(s) in question are given by (12) and (15) of \cite{Le15}. For $k=1$ and $\alpha$ as given in (18) of \cite{Le15}, let us change coordinate to $\hat{x} = x + \alpha$ (with range $(a+\alpha, b+\alpha)$).  Then we have
\begin{equation}\label{LeBrun}
g=\hat{x} \left[\frac{d\hat{x}^2}{2\Psi} + 2(\sigma_1^2 + \sigma_2^2)\right] + \frac{2\Psi}{\hat{x}}\sigma_3^2
\end{equation}
with
$$ \Psi = \frac{((\alpha + b) -\hat{x})(\hat{x} - (\alpha +a))}{b-a}\left[ \hat{x} + E ((\alpha + b) -\hat{x})(\hat{x} - (\alpha +a))\right].$$
Now without loss we set $u-v$ in (19) of \cite{Le15} equal to $2\pi$ and hence
have $a$ and $b$ of \cite{Le15} given as follows
$$a= \frac{1}{2\frak{z}}$$
and
$$b= (1+2\frak{z})( \frac{1}{2\frak{z}}) = 1+ \frac{1}{2\frak{z}},$$
where, for this remark only, $\frak{z}$ refers to page 31 of \cite{Le15} and not the moment map coordinate from our text.
As pointed out in \cite{Le15}, any value $\frac{u}{v}>9$, arises from two different values of $\frak{z}>0$ that are reciprocals of each other and this is how $h_1$ and $h_2$ arise.

Now we make the following observations:
With $a$ and $b$ given in terms of
$\frak{z}$ as above, we have that
$$a+\alpha = a-\frac{4a^2b}{(a+b)^2} = \frac{\frak{z}}{2(\frak{z}+1)^2}$$
is invariant under $\frak{z} \mapsto 1/\frak{z}$ and hence so is $b+\alpha = a+\alpha+1$.
Finally, we observe from (16) of \cite{Le15} that
$$E =  \frac{\alpha - 2 a}{a^2 + 4 a b + b^2} = \frac{\frak{-z}}{(\frak{z}+1)^2},$$
and so this is also invariant under $\frak{z} \mapsto 1/\frak{z}$.  Thus $g$ in \eqref{LeBrun} is  geometrically the same for $\frak{z}$ and $1/\frak{z}$.
\end{remark}

\section{The Einstein-Hilbert functional}\label{EHsection}

\subsection{The admissible K\"ahler classes revisited}\label{classes}
Each admissible ruled surface $S_n$, $n\in {\mathbb Z}^+$, belong to one of only two possible diffeomorphism types; the product $S^2 \times \Sigma$ or the unique twisted $S^2$-bundle over $\Sigma$, $S^2 \tilde{\times} \Sigma$  . In fact, if $n$ is even, $S_n$ is of the former type, and if $n$ is odd, $S_n$ is of the latter type. Although not admissible (in the sense above), we will use $S_0$ to mean the trivial ruled surface ${\mathbb C}{\mathbb P}^1 \times \Sigma$.

We will now express the (admissible) K\"ahler class from \eqref{class1} in a basis that only depends on the parity of $n$. Let $n=2k$, when $n$ is even and $n=2k+1$, when $n$ is odd. Further, let $E$ denote (the Poincar\'e dual) of the section of $S_n \rightarrow \Sigma$ which has self-intersection zero, in the case of $n$ being even, and self-intersection one in the case of $n$ being odd. In fact, we may write  $E= E_0-k C$

It is now standard procedure to show that \eqref{class1} implies that
\begin{equation}\label{class2}
[\omega] = 4\pi (E +g(x,k)\, C),
\end{equation}
where
$$g(x,k)= \left\{ \begin{array}{cl}
\frac{k}{x}, & \text{when}\,n=2k\,\text{is even}\\
\\
\frac{2k+1-x}{2x}, & \text{when}\,n=2k+1\,\text{is odd}\\
\end{array}\right. $$
Note that if we write the admissible K\"ahler class as
$$\Omega = 4\pi(E+ p\,C),$$ then
$$x =  \left\{ \begin{array}{cl}
\frac{k}{p}, & \text{when}\,n=2k\,\text{is even}\\
\\
\frac{2k+1}{2p+1}, & \text{when}\,n=2k+1\,\text{is odd}\\
\end{array}\right. $$

Now we see that for a fixed $p$, the cohomology class $\Omega = 4\pi(E+ p\,C)$  on the smooth manifold  $S^2 \times \Sigma$ is  an K\"ahler class on
the admissible ruled surface $S_{2k}$ iff $1\leq k<p$. In that case, $\Omega$ has the form of \eqref{class1} with $x=k/p$. In the case, where the smooth manifold is instead $S^2 \tilde{\times} \Sigma$, $\Omega$ is a K\"ahler class on $S_{2k+1}$ iff $0\leq k<p$. In that case, $\Omega$ has the form of \eqref{class1} with  $x=(2k+1)/(2p+1)$.

\begin{remark}\label{productmetric}
Note that in the case where the smooth manifold  is $S^2 \times \Sigma$, $\Omega$ is a K\"ahler class on $S_0$ as long as $p>0$. In fact, $E$ may be viewed as
the cohomology class $[\omega_{FS}]$ and $C$ may be viewed as the cohomology class $[\hat{\omega_\Sigma}]$ where $\omega_{FS}$ denotes the unit volume Fubini-Study metric (with constant scalar curvature $8\pi$) and $\hat{\omega_\Sigma}$ denotes the unit volume constant scalar curvature metric on $\Sigma$ (with constant scalar curvature $8\pi(1-\mathfrak{g})$).

If we consider the K\"ahler CSC (and Einstein-Maxwell solution) product metric on $S_0$ corresponding to representation $4\pi(\omega_{FS}+p \hat{\omega_\Sigma})$ of $\Omega$, we can easily calculate the value of ${\mathfrak{S}}$ to be
$$Scal\, \sqrt{Vol} = \frac{8\pi\left(1+\frac{( 1-\mathfrak{g})}{p}\right)}{4\pi} \sqrt{(4\pi)^2 p} = 8\pi \sqrt{p}\left(1+\frac{(1-\mathfrak{g})}{p}\right).$$

Since $8\pi \sqrt{p}\left(1+\frac{(1-\mathfrak{g})}{p}\right) > 8\pi\sqrt{6}$ for $p$ sufficiently large, we obviously have many examples of these K\"ahler CSC product metrics that are NOT Yamabe minimizers in their conformal class. On the other hand, for $\mathfrak{g} \geq 2$ we can pick values of $p>0$ such that  $Scal\, \sqrt{Vol} \leq 0$, so some of these product metrics ARE Yamabe minimizers.
\end{remark}

\subsection{The Einstein-Hilbert Functional applied to the admissible Einstein-Maxwell solutions}

For our solutions $h$ from Section \ref{admisEMsoln} with constant $Scal(h)$, the Einstein-Hilbert functional equals
$Scal(h)\, Vol(h)^{1/2}$ where
$$Vol(h) = \int_{S_n} d\mu_h.$$
Since $d\mu_h = (\mathfrak{z}+b)^{-4} d\mu_g$, where $g$ is the admissible K\"ahler metric in questions, we have
$$
\begin{array}{ccl}
Vol(h) & = & \int_{S_n} (\mathfrak{z}+b)^{-4} \omega^2/2 \\
\\
& =  & \int_{S_n} (\mathfrak{z}+b)^{-4}(\mathfrak{z} + 1/x) \omega_\Sigma \wedge d\mathfrak{z} \wedge \theta \\
\\
& = & 2\pi Vol(\omega_\Sigma) \int_{-1}^1 (\mathfrak{z}+b)^{-4}(\mathfrak{z} + 1/x) \, d\mathfrak{z}\\
\\
& = &( 2\pi)^2 n  \int_{-1}^1 (\mathfrak{z}+b)^{-4}(\mathfrak{z} + 1/x) \, d\mathfrak{z}\\
\\
& =& ( 2\pi)^2 n  \frac{2(3b^2- 4 b x +1)}{3x(b^2-1)^3 }.
\end{array}
$$
Using \eqref{A} we then get
\begin{equation}\label{EHfunc1}
\begin{array}{cl}
&Scal(h) Vol(h)^{1/2} \\
\\
= & \frac{12\pi\sqrt{n} \left(1-6 b^2+b^4+2 b x+2 b^3 x-s_\Sigma x+b^4 s_\Sigma x\right)}{3 b^2-1}\sqrt{  \frac{2(3b^2- 4 b x +1)}{3x(b^2-1)^3 }}.
\end{array}
\end{equation}

Since $h$ has constant scalar curvature, the Yamabe constant of $[h]$ must satisfy that $Y_{[h]} \leq Scal(h) Vol(h)^{1/2}$.

On the other hand, for each $x\in (0,1)$, the CSC metric $h$ is conformal to an (admissible) K\"ahler metric $g$ in the class
given by \eqref{class1} (or \eqref{class2}). LeBrun's work \cite{Le97} then implies that
$$Y_{[h]} = Y_{[g]} \leq \frac{4\pi c_1 \cdot [\omega]}{\sqrt{[\omega]^2/2}},$$
where equality happens if and only if the K\"ahler metric is actually a Yamabe minimizer in $[g]$.
Since we know that $g$ is not even CSC here\footnote{In fact, there are no CSC K\"ahler metrics on $S_n$ \cite{at}}, we know that in our present case the inequality above is sharp.

Using the formulas on pages 564--565 in \cite{acgt}, we easily calculate that for any genus $\mathfrak{g}$ and any admissible K\"ahler metric $g$ in $[\omega]$ we have
$$\frac{4\pi c_1 \cdot [\omega]}{\sqrt{[\omega]^2/2}} =   \frac{\int_{S_n} Scal(g) \, d\mu_g}{\sqrt{\int_{S_n} d\mu_g}}= \frac{4 \pi (2+2 s_\Sigma x) \sqrt{ n}}{\sqrt{2x}} .$$
Thus, we have for the moment the rough estimate
\begin{equation}\label{KYestimate}
Y_{[h]} <  \frac{4 \pi (2+2 s_\Sigma x) \sqrt{ n}}{\sqrt{2x}}.
\end{equation}
We shall see below that \eqref{EHfunc1}  and the fact that $Y_{[h]} < Scal(h) Vol(h)^{1/2}$ will improve this estimate a bit in some cases. Of course \eqref{KYestimate} is only an improvement over Aubin's estimate as long as
$ \frac{4 \pi (2+2 s_\Sigma x) \sqrt{ n}}{\sqrt{2x}} < 8\pi\sqrt{6}$.
Likewise, if $Scal(h) Vol(h)^{1/2} > 8\pi\sqrt{6}$, then all we can say is that $h$ is NOT a Yamabe minimizer of $[h]$.

\subsection{First Hirzebruch surface:}
The Einstein-Hilbert functional has already been treated thoroughly for Hirzebruch surfaces in \cite{Le15}, so we shall not treat this case in general. We will restrict ourselves a
quick discussion of the interesting Case 2 solutions in Section \ref{Case2} on the first Hirzebruch Surface. Here we have two different CSC metrics $h_1$ and $h_2$ in the same conformal class.

For a given $x\in (4/5,1)$, we substitute $(x-1) b^2+ x b  - x =0$, $s_\Sigma =2$, and $n=1$ into \eqref{EHfunc1} to get
\begin{equation}
Scal(h_i) Vol(h_i)^{1/2} = 4 \pi \sqrt{6} \sqrt{\frac{4x-1}{x}}, \quad i=1,2.
\end{equation}
Apparently $Scal(h) Vol(h)^{1/2}$ has the same value for $h_1$ and $h_2$. This fact may also be verified directly from equation (22) of \cite{Le15}.
As observed in \cite{Le15}, as $x\rightarrow 1$, we have that $$Scal(h_i) Vol(h_i)^{1/2}  \rightarrow 12 \pi \sqrt{2},$$ which is the value of
$Scal(h) Vol(h)^{1/2}$ for the standard Fubini-Study metric on ${\mathbb C}{\mathbb P}^2$. Notice that  the right hand side of \eqref{KYestimate}, which in this case is
$ \frac{4 \pi \sqrt{2} (1+2 x)}{\sqrt{x}}$, has the same limit for
$x \rightarrow 1$.

We observe that for $x\in (4/5,1)$
$$4 \pi \sqrt{6} \sqrt{\frac{4x-1}{x}} <   \frac{4 \pi \sqrt{2} (1+2 x)}{\sqrt{x}} < 8\pi \sqrt{6},$$ so
$$ Y_{[h_1]}=   Y_{[h_2]} \leq 4 \pi \sqrt{6} \sqrt{\frac{4x-1}{x}}$$
gives an improved estimate in comparison to \eqref{KYestimate}. This does not necessarily imply that $h_1$ and $h_2$ are Yamabe minimizers, but in any case we have a new example showing the non-uniqueness of CSC metrics in a conformal class.

\subsection{Case where of $\Sigma = T^2$:}
This is the case where $\mathfrak{g} = 1$ and $s_\Sigma=0$. Since $\mathfrak{g} <2$, we know from Section \ref{Case1} that we have genuine solutions for any value of $x \in (0,1)$. For a given $x\in (0,1)$, we substitute \eqref{Case1soln} and $s_\Sigma =0$ into \eqref{EHfunc1} to get
\begin{equation}\label{EHfunc2}
Scal(h) Vol(h)^{1/2} = 4 \sqrt{6} \sqrt{n} \pi  \sqrt{\frac{1-x^2}{x \left(1+2 \sqrt{1-x^2}\right)}}.
\end{equation}
We will apply \eqref{EHfunc2} in two different settings below:

\subsubsection{The Yamabe constant for fixed $S_n$:}
Assuming that $n$ is fixed for the moment,
one may check that this is a monotone decreasing function over the interval $(0,1)$ and that $$\lim_{x \rightarrow 0} Scal(h) Vol(h)^{1/2} = +\infty$$ while
$$\lim_{x \rightarrow 1} Scal(h) Vol(h)^{1/2} = 0.$$ Since for $x>0$ sufficiently small we will then have $Scal(h) Vol(h)^{1/2} > 8\pi\sqrt{6}$ we can conclude that the CSC metrics $h$ are NOT Yamabe minimizers for $x>0$ sufficiently small. Whether they are ever Yamabe minimizers remains unknown.

In the present case ($s_\Sigma =0$), \eqref{KYestimate} becomes
$$Y_{[h]} <  \frac{8 \pi \sqrt{ n}}{\sqrt{2x}} .$$
Not surprisingly, the right hand side of \eqref{EHfunc2} is less than $ \frac{8 \pi \sqrt{ n}}{\sqrt{2x}}$ for all $x\in (0,1)$. Since $Scal(h) >0$ here, we do not know if any of the CSC metrics $h$ are in fact Yamabe minimizers, but their Einstein-Hilbert functional values do offer an improvement in the estimate of $Y_{[h]}$. Notice that as $x\rightarrow 0$, the difference between the two estimates approaches zero, whereas when
$x\rightarrow 1$, the difference approaches $4 \pi \sqrt{n/2}$.

\subsubsection{The Einstein-Hilbert functional on $\mathscr{G}_\Omega$:}\label{EHg=1}

Let us now fix  the cohomology class $\Omega = 4\pi(E+ p\,C)$  on the smooth manifold  $S^2 \times T^2$. Then, using the discussion in Section \ref{classes}, we get an
Einstein-Maxwell solution from Section \ref{Case1} $h_k$ in $\mathscr{G}_\Omega$ for each choice of $k=1,...,\lceil p \rceil -1$ where $x=k/p$. Using \eqref{EHfunc2} we have
\begin{equation}\label{EHfunc3}
Scal(h_k) Vol(h_k)^{1/2} = 8\pi  \sqrt{3}   \sqrt{\frac{p^2-k^2}{p+2 \sqrt{p^2-k^2}}}.
\end{equation}
Note that if we formally substitute $k=0$ into the right hand side of \eqref{EHfunc3} we get the value of $Scal(h_0) Vol(h_0)^{1/2}$, when $h_0$ is the CSC K\"ahler product metric discussed in Remark \ref{productmetric}. Since $h_0$ is also an Einstein-Maxwell solution we may say that for each choice of $k=0,...,\lceil p \rceil -1$ we have an
Einstein-Maxwell solution with metric $h_k$ such that $Scal(h_k) Vol(h_k)^{1/2}$ is given by \eqref{EHfunc3}.

It is easy to confirm that the right hand side of  \eqref{EHfunc3} is strictly decreasing for $k=0,...,\lceil p \rceil -1$ and in particular we observe $\lceil p \rceil$ different values of
${\mathfrak S}$ for the fixed $\mathscr{G}_\Omega$.  Note that we could also consider  CSC local product K\"ahler metrics on ruled surfaces of the type ${\mathbb P}({\mathcal O} \oplus {\mathcal L}_0) \rightarrow T^2$, where $L_0$ is a non-trivial line bundle of degree zero, but the value of $Scal(h) Vol(h)^{1/2}$ would simply duplicate the value for
the product metric from Remark \ref{productmetric}, so we will not pursue this any further.

Next we fix the cohomology class $\Omega = 4\pi(E+ p\,C)$  on the smooth manifold $S^2 \tilde{\times} T^2$. Imitating the steps above we get an Einstein-Maxwell solution
$h_k$ in $\mathscr{G}_\Omega$ for each choice of $k=0,...,\lceil p \rceil -1$ where \newline
$x=(2k+1)/(2p+1)$ and now
\begin{equation}\label{EHfunc4}
Scal(h_k) Vol(h_k)^{1/2} = 4\pi  \sqrt{6}   \sqrt{\frac{(2p+1)^2-(2k+1)^2}{(2p+1)+2 \sqrt{(2p+1)^2-(2k+1)^2}}}.
\end{equation}
Again, we can confirm that the right hand side of  \eqref{EHfunc4} is strictly decreasing for $k=0,...,\lceil p \rceil -1$ and hence - also in this case - we display $\lceil p \rceil$ different values of ${\mathfrak S}$ for the fixed $\mathscr{G}_\Omega$.

\begin{remark} On $S^2 \tilde{\times} T^2$ there exists an additional complex structure from the ${\mathbb C}{\mathbb P}^1$ bundle over $T^2$ of the form
${\mathbb P}({\mathcal E}) \rightarrow \Sigma$, where ${\mathcal E} \rightarrow$ is a rank two indecomposable stable holomorphic vector bundle \cite{atiyah55, atiyah57, suwa}.
Without loss we assume that the degree of ${\mathcal E}$ is one and so $E$ may be viewed as the zero section on ${\mathbb P}({\mathcal E}) \rightarrow \Sigma$.
For this complex structure $\Omega = 4\pi(E+ p\,C)$ is a K\"ahler class if and only if $p > -1/2$ \cite{fujiki}. Moreover, due to the stability of $E$ we get a local
product CSC metric $\tilde{h}_0$. Indeed, by recognizing that $E$ is given by ``$e+f/2$'' from \cite{fujiki}, where $f = C$, and using Lemma 1 of \cite{fujiki}, we get that for this local product CSC metric we have
$$Scal(\tilde{h}_0) Vol(\tilde{h}_0)^{1/2} = 4\pi\sqrt{2}\sqrt{2p+1}.$$
Now this is in fact the limit of the right hand side of \eqref{EHfunc4} as $k \rightarrow -1/2$, i.e. $2k+1 \rightarrow 0$.

\end{remark}

Using Proposition 2 of \cite{Le15} we may conclude that Theorem E of \cite{Le15} also holds for $S^2 \times T^2$ and the twisted $S^2$ bundle over $T^2$.

\subsection{Case where $\Sigma$ has genus at least two:}
This is the case where $\mathfrak{g} \geq 2$ and $s_\Sigma=2(1-\mathfrak{g})/n <0$.
For a given $x\in (0,1)$, \emph{satisfying that $F$ from \eqref{F} with the associated $c$ from \eqref{c} satisfies (i) of
\eqref{positivityF}}, we substitute \eqref{Case1soln} and $s_\Sigma=2(1-\mathfrak{g})/n$ into \eqref{EHfunc1} to get
\begin{equation}\label{EHfunc5}
\begin{array}{ccl}
Scal(h) Vol(h)^{1/2} & = & 4 \sqrt{6} \sqrt{n} \pi  \sqrt{\frac{1-x^2}{x \left(1+2 \sqrt{1-x^2}\right)}} \\
\\
& + &  8 \pi \sqrt{6/n}  (1-\mathfrak{g})\sqrt{\frac{x}{(1+2 \sqrt{1-x^2})}}.
\end{array}
\end{equation}

\subsubsection{The Yamabe constant for fixed $S_n$:}
Assuming that $n$ and $\mathfrak{g}$ are fixed for the moment,
one may check that this is a monotone decreasing function over the interval $(0,1)$ and that $$\lim_{x \rightarrow 0} Scal(h) Vol(h)^{1/2} = +\infty$$ while
$$\lim_{x \rightarrow 1} Scal(h) Vol(h)^{1/2} <0.$$ For $x>0$ sufficiently small we will then have $Scal(h) Vol(h)^{1/2} > 8\pi\sqrt{6}$ as well as
$F$ from \eqref{F} with the associated $c$ from \eqref{c} satisfying (i) of
\eqref{positivityF}.  Thus we can conclude, also in the case, that we have CSC metrics $h$ that are NOT Yamabe minimizers for $x>0$ sufficiently small.

On the other hand, the fact that $\lim_{x \rightarrow 1} Scal(h) Vol(h)^{1/2} <0$ is only a formal observation, since we know that for $0<x<1$ sufficiently close to $1$,
(i) of \eqref{positivityF} fails. We will therefore restrict ourselves to an example.

\begin{example}
Revisiting Example \ref{ex1} we assume that that $n=1$ and the genus of $\Sigma$, $\mathfrak{g}=2$. Then
\begin{equation}\label{EHfunc6}
Scal(h) Vol(h)^{1/2} = 4 \sqrt{6} \pi  \left( \sqrt{\frac{1-x^2}{x \left(1+2 \sqrt{1-x^2}\right)}} - 2 \sqrt{\frac{x}{(1+2 \sqrt{1-x^2})}}\right),
\end{equation}
which reaches negative values by (approximatively) $x=0.44722$ and safely before $x= x_{s_\Sigma,2} \approx 0.97367$.
When $s_\Sigma =-2$, \eqref{KYestimate} becomes
$$Y_{[h]} < \frac{8 \pi (1-2 x)}{\sqrt{2x}} .$$
As in the case of $\Sigma =T^2$ we have that the right hand side of this inequality is larger than the right hand side of \eqref{EHfunc6} and thus
$Scal(h) Vol(h)^{1/2}$ offers an improvement of the estimate of $Y_{[h]}$ also in this case. In particular, it is interesting to notice the interval (approx. $(0.44722, 0.5)$) where
$ \frac{8 \pi (1-2 x)}{\sqrt{2x}} >0$ and $Scal(h) Vol(h)^{1/2} <0$. Here we have  that the estimate \eqref{KYestimate} does not predict the fact that the conformal class has a constant scalar curvature Yamabe minimizer of negative scalar curvature.
\end{example}

\subsubsection{The Einstein-Hilbert functional on $\mathscr{G}_\Omega$:}\label{EHgengeq2}
Due to the existence issues of our special Einstein-Maxwell solutions in the higher genus case, we have to be a bit more careful here. However, as the argument below will show, we can still confirm that Theorem E of \cite{Le15} also holds for $S^2 \times \Sigma$ as well as the twisted bundle $S^2\tilde{\times}\Sigma$.

Consider an arbitrary value $K\in {\mathbb N}$ and let
$p_K$ be such that
\begin{equation}\label{pk}
\forall k\in \{ 1,...,K\}, \quad p_K > \frac{(1-\mathfrak{g})^2}{k} + 2k.
\end{equation}
In particular, note that $p_K >K$.
Then let us fix  the cohomology class $\Omega = 4\pi(E+ p_K\,C)$  on the smooth manifold  $S^2 \times \Sigma$.
For any $k = 1,....,K$ we know that from Section \ref{classes} that $\Omega$ is  an K\"ahler class on
the admissible ruled surface $S_{2k}$ and $\Omega$ has the form of \eqref{class1} with $x=k/p_K$.
Since $\frac{(1-\mathfrak{g})^2}{k}= s_\Sigma$, we have from \eqref{pk} that
$$x = k/p_K < \frac{1}{ \left( \frac{(1-\mathfrak{g})}{k}\right)^2 + 2} = \frac{1}{s_\Sigma^2 +2}$$
and so, by Lemma \ref{littlelemma}
we get an
Einstein-Maxwell solution from Section \ref{Case1} $h_k$ in $\mathscr{G}_\Omega$. Using \eqref{EHfunc5} we have
\begin{equation}\label{EHfunc7}
\begin{array}{rcl}
Scal(h_k) Vol(h_k)^{1/2} & = & 8\pi  \sqrt{3}   \sqrt{\frac{p_K^2-k^2}{p_K+2 \sqrt{p_K^2-k^2}}}  \\
\\
&+ &  8 \pi \sqrt{3}  (1-\mathfrak{g})\sqrt{ \frac{1}{p_K+2 \sqrt{p_K^2-k^2}}}.
\end{array}
\end{equation}
We observe that the right hand side of  \eqref{EHfunc7} is strictly decreasing for $k=1,...,K$ and in particular we observe $K$ different values of
${\mathfrak S}$ for the fixed $\mathscr{G}_\Omega$.

Consider an arbitrary value $K\in {\mathbb N}$ but now let
$\tilde{p}_K$ be such that
\begin{equation}\label{tildepk}
\forall k\in \{ 0,...,K\}, \quad 2\tilde{p}_K +1 > \frac{2(1-\mathfrak{g})^2}{2k+1} + 2(2k+1) .
\end{equation}
In particular, note that $\tilde{p}_K >K$.
Then we fix the cohomology class $\Omega = 4\pi(E+ \tilde{p}_K\,C)$  on the smooth manifold $S^2 \tilde{\times} \Sigma$.
For any $k = 0,....,K$ we know that from Section \ref{classes} that $\Omega$ is  an K\"ahler class on
the admissible ruled surface $S_{2k+1}$ and $\Omega$ has the form of \eqref{class1} with $x=\frac{2k+1}{2\tilde{p}_K+1}$.
Since $\frac{2(1-\mathfrak{g})^2}{2k+1}= s_\Sigma$, we have from \eqref{tildepk} that
$$x =\frac{2k+1}{2\tilde{p}_K+1} < \frac{1}{ \left( \frac{2(1-\mathfrak{g})}{2k+1}\right)^2 + 2} = \frac{1}{s_\Sigma^2 +2}$$
and so, by Lemma \ref{littlelemma}
we get an
Einstein-Maxwell solution from Section \ref{Case1} $h_k$ in $\mathscr{G}_\Omega$ for each choice of  $k=0,...,K$ where $x=(2k+1)/(2\tilde{p}_K+1)$. In this case,
\begin{equation}\label{EHfunc8}
\begin{array}{ccl}
Scal(h_k) Vol(h_k)^{1/2} & = & 4\pi  \sqrt{6}   \sqrt{\frac{(2\tilde{p}_K+1)^2-(2k+1)^2}{(2\tilde{p}_K+1)+2 \sqrt{(2\tilde{p}_K+1)^2-(2k+1)^2}}} \\
\\
&+ & 8 \pi \sqrt{6}  (1-\mathfrak{g}) \sqrt{\frac{1}{1+2 \tilde{p}_K+4\sqrt{\tilde{p}_K(1+\tilde{p}_K)-k (1+k)}}}.
\end{array}
\end{equation}
Again, we can confirm that the right hand side of  \eqref{EHfunc8} is strictly decreasing for $k=0,...,K$ and hence  in this case we display $K+1$ different
values of ${\mathfrak S}$ for the fixed $\mathscr{G}_\Omega$. This observation, together with the similar conclusion from Section \ref{EHg=1}, proves Theorem \ref{modulispacetheorem}.

\section{Addendum: The Apostolov-Maschler Futaki Invariant and Stability}\label{amf}
After we submitted our paper for publication, Vestislav Apostolov and Gideon Maschler posted a beautiful paper \cite{am} developing the general theory of
conformally K\"ahler Einstein-Maxwell metrics. These metrics generalize strongly hermitian Einstein-Maxwell solutions to higher dimensions.
Similarly to what Donaldson and Fujiki did for the theory of constant scalar curvature \cite{Don97,fujiki2},
they put the theory in a moment-map setting.

We will now give a short discussion on how our examples
from Section \ref{admisEMsoln} might fit into this general picture.

Observe that the system of equations \eqref{c} and \eqref{b} of Section \ref{admisEMsoln} is equivalent to the system consisting of
\begin{equation}\label{newc}
\begin{array}{ccl}
c & = & \frac{-1+3 b x-s_\Sigma x}{2 \left(3 b^2-1\right)} \\
\\
&+&  \frac{3x\left(b^2x -2 b +x \right) \left((s_\Sigma x-2) b^2+2 b x-s_\Sigma x \right)}{2(3b^2-1)((b^2-3) x^2 + 4 b x + (1-3b^2))}
\end{array}
\end{equation}
together with \eqref{b}. Note that \eqref{newc} is well-defined for all $0<x<1$ and all $|b| >1$.

If we substitute \eqref{newc} into \eqref{F} from Section \ref{admisEMsoln} we obtain a polynomial $F_b({\mathfrak z})$ for each choice of $|b|>1$ and $0<x<1$.
This polynomial clearly satisfies the end point conditions of \eqref{positivityF}.
\underline{As long as (i) of \eqref{positivityF} is satisfied} (which is very much a non-trivial assumption for higher genera of $\Sigma$), we then have an admissible K\"ahler metric  and one can check directly that for $h = ({\mathfrak z} +b)^{-2} g$, the scalar curvature is an affine function of ${\mathfrak z}$, i.e., a killing potential.
So as  Apostolov and Maschler points out, $F_b({\mathfrak z})$ is an analogue of the so-called {\it extremal polynomial} used in the theory of extremal K\"ahler metrics.
Further, \eqref{b} will then be equivalent with the vanishing of the Apostolov-Maschler Futaki invariant as defined in Section 2 of \cite{am}.

In fact, similarly to the calculations in Section 5.3 of \cite{am}, we may take any convenient function satisfying all the conditions of \eqref{positivityF} (e.g. $F(z) = (1-{\mathfrak z}^2)(1+x {\mathfrak z})$) and, using the formulae in Corollary 1 in \cite{am}, calculate the Apostolov-Maschler Futaki invariant. From that we easily discover that \eqref{b} is equivalent with the vanishing of the Apostolov-Maschler Futaki invariant even if $F_b({\mathfrak z})$ does not satisfy (i) of \eqref{positivityF}. The consequence of this is then that whenever the genus of $\Sigma$ is at least one, \eqref{b} gives all the possible potentials ${\mathfrak z} +b$ for which there is an $\omega$-compatible K\"ahler metric (not necessarily given by the admissible ansatz) which is conformal to an Einstein-Maxwell metric with the conformal factor $({\mathfrak z} + b)^{-2}$.

When \eqref{b} holds but $F_b({\mathfrak z})$ does not satisfy (i) of \eqref{positivityF} we suspect that an appropriate notion of $K$-polystability (as developed in the toric case in \cite{am} - see also the discussion in Section 5.1 of \cite{am}) fails. It is tempting to conjecture that for any such case there is no $\omega$-compatible K\"ahler metric which is conformal to an Einstein-Maxwell metric with the conformal factor
$({\mathfrak z} + b)^{-2}$.

In any case, we hope that similarly to what happened in the extremal K\"ahler metrics story (see e.g. \cite{gs}), the fact that our ansatz fails for certain K\"ahler classes in the higher genus case in Section \ref{admisEMsoln} will provide examples that will guide the further development of $K$-polystability for
conformally K\"ahler, Einstein-Maxwell metrics on K\"ahlerian complex compact manifolds in general.

\end{document}